\documentclass[11pt,reqno]{amsart}
\usepackage{a4wide}
\usepackage{hyperref}
\usepackage{color}
\usepackage{cite}
\usepackage{amsfonts,amssymb}
\usepackage{pb-diagram}
\usepackage[all]{xy}
\usepackage{graphicx,color}
\usepackage{mathrsfs}
\usepackage{amstext}
\usepackage{sseq}

\definecolor{green}{rgb}{0,0.6,0.2}
\definecolor{blue}{rgb}{0,0,1}

\hypersetup{colorlinks,
linkcolor=blue,
filecolor=green,
citecolor=green}

\theoremstyle{plain}

\newtheorem{neu}{}[section]
\newtheorem{Cor}[neu]{Corollary}
\newtheorem*{Cor*}{Corollary}
\newtheorem{Thm}[neu]{Theorem}
\newtheorem*{Thm*}{Theorem}
\newtheorem{Prop}[neu]{Proposition}
\newtheorem*{Prop*}{Proposition}
\theoremstyle{definition}
\newtheorem{Lemma}[neu]{Lemma}
\newtheorem*{Rmk*}{Remark}
\newtheorem{Rmk}[neu]{Remark}

\newtheorem*{Ex*}{Example}

\newtheorem*{Qu*}{Question}

\theoremstyle{remark}

\theoremstyle{definition}

\newcommand{\x}{\times}

\newcommand{\p}{\partial}

\newcommand{\Om}{\Omega}
\newcommand{\into}{\hookrightarrow}

\newcommand{\N}{{\mathbb{N}}}
\newcommand{\Z}{{\mathbb{Z}}}
\newcommand{\R}{{\mathbb{R}}}

\newcommand{\Q}{{\mathbb{Q}}}

\newcommand{\Id}{\mathrm{Id}}

\newcommand{\Fix}{\mathrm{Fix}\,}

\newcommand{\beq}{\begin{equation}}
\newcommand{\beqn}{\begin{equation}\nonumber}
\newcommand{\eeq}{\end{equation}}

\newcommand{\bea}{\begin{equation}\begin{aligned}}
\newcommand{\bean}{\begin{equation}\begin{aligned}\nonumber}
\newcommand{\eea}{\end{aligned}\end{equation}}

\numberwithin{equation}{section}
\numberwithin{figure}{section}

\begin{document}
\title{On reversible maps and symmetric periodic points}
\author{Jungsoo Kang}

\address{Mathematisches Institut,
Westf\"alische Wilhelms-Universit\"at M\"unster, M\"unster, Germany}
\email{jungsoo.kang@me.com,\;jkang\_01@uni-muenster.de}

\maketitle

\begin{abstract}
In reversible dynamical systems, it is frequently of importance to understand symmetric features. The aim of this paper is to explore symmetric periodic points of reversible maps on planar domains invariant under a reflection. We extend Franks' theorem on a dichotomy of the number of periodic points of area preserving maps on the annulus to symmetric periodic points of area preserving reversible maps. Interestingly, even a non-symmetric periodic point guarantees infinitely many symmetric periodic points. We prove an analogous statement for symmetric odd-periodic points of area preserving reversible maps isotopic to the identity, which can be applied to 
dynamical systems with double symmetries. Our approach is simple, elementary and far from Franks' proof. We also show that a reversible map has a symmetric fixed point if and only if it is a twist map which generalizes a boundary twist condition on the closed annulus in the sense of Poincar\'e-Birkhoff.  Applications to symmetric periodic orbits in reversible dynamical systems with two degrees of freedom are briefly  discussed.
\end{abstract}

\setcounter{tocdepth}{1}
\tableofcontents

\section{Main results and applications}
The study on symmetric features of reversible dynamical systems is of great importance. This paper is majorly concerned with symmetric periodic points of reversible maps on planar domains, which can be applied to find symmetric periodic orbits of reversible dynamical systems with two degrees of freedom possessing global surfaces of section invariant under symmetries, see below. Among other studies on symmetric periodic points, we refer the reader to \cite{Bir15,DeV54,Dev76} which are related to our approach. \\[-1.5ex]

Let $I$ be the reflection on $\R^2$ given by
$$
I:\R^2\to\R^2,\quad (x,y)\mapsto(-x,y).
$$
A connected planar domain $\Omega\subset\R^2$ is said to be {\bf invariant} if $I(\Omega)=\Omega$. Then $I$ descends to a reflection $I_\Omega$ on an invariant domain $\Omega$. A homeomorphism $f$ on an invariant domain $(\Omega,I_\Om)$ obeying  
$$
f\circ I_\Omega=I_\Omega\circ f^{-1}
$$
is called a {\bf reversible map}. A point $z\in\Omega$ is called a {\bf symmetric periodic point} of $f$ if 
$$
f^k(z)=z,\quad f^\ell(z)=I_\Omega(z) \quad\textrm{for some }\; k,\,\ell\in\N.
$$
The minimal number $k\in\N$ satisfying the requirement is called a {\bf period} of a symmetric fixed point $z$ of $f$. 
In particular if $k=\ell=1$, $z$ is called a {\bf symmetric fixed point}. Note that symmetric fixed points lie on the fixed locus $\Fix I_\Omega$ of $I_\Omega$. Recall that a point meeting the first condition is called a periodic point of $f$. Throughout this paper, we always assume that $\Omega\subset\R^2$ is an invariant connected domain with a reflection $I_\Om$ and the following invariant domains are of interest to us.
$$
A:=\{z\in\R^2\,|\,1\leq|z|\leq2\},\quad S:=\{(x,y)\in\R^2\,|\,0\leq|y|\leq1\},\quad D:=\{z\in\R^2\,|\,|z|\leq 1\}
$$
and 
$$
 \mathring A:=\{z\in\R^2\,|\,1<|z|<2\},\quad \mathring D:=\{z\in\R^2\,|\,|z|< 1\}.
$$

A beautiful theorem by Franks \cite{Fra92,Fra96} states that every area preserving homeomorphism on $A$ or $\mathring A$ with a periodic point has infinitely many interior periodic points. Our first result is to extend this dichotomy to  symmetric periodic points of reversible maps.

\begin{Thm}\label{thm:Franks-like theorem}
Every area preserving reversible map on $A$ or $\mathring A$ is either periodic point free or has infinitely many interior  symmetric periodic points.
\end{Thm}

It is worth emphasizing that  even a {\em non-symmetric} periodic point also guarantees infinitely many interior symmetric periodic points. Although the theorem is stated for $A$ and $\mathring A$, it carries over to the half-closed annuli $\{z\in\R^2\,|\,1\leq|z|<2\}$ and $\{z\in\R^2\,|\,1<|z|\leq2\}$. We will show in Proposition \ref{prop:symmetric fixed point} that every area preserving reversible map on $D$ or $\mathring D$ has at least one interior symmetric fixed point. The following corollary is a direct consequence of these results.

\begin{Cor}\label{cor:dichotomy}
Every area preserving reversible map on $D$ or $\mathring D$ has either precisely one interior symmetric fixed point with no other periodic points or infinitely many interior symmetric periodic points. 
\end{Cor}

The following theorem improves Theorem \ref{thm:Franks-like theorem} for reversible maps isotopic to the identity under an additional assumption on the parity of periods. Odd-periodic points have an intimate relation with centrally symmetric periodic orbits as described in the next subsection. 

\begin{Thm}\label{thm:dichotomy2}
Any area preserving reversible map on $A$ or $\mathring A$ isotopic to the identity with an odd-periodic point (not necessarily symmetric nor interior) has   infinitely many interior symmetric odd-periodic points. In consequence, every area preserving orientation preserving reversible map on $D$ or $\mathring D$ has either precisely one interior symmetric fixed point with no other odd-periodic points or infinitely many interior symmetric odd-periodic points. 
\end{Thm}

In view of the Poincar\'e-Birkhoff theorem, we expect that there exists a symmetric fixed point on $A$ under a boundary twist condition. Recall that a homeomorphism $f:A\to A$ isotopic to the identity is said to satisfy the {\bf boundary twist condition} if $f$ rotates two boundary circles of $A$ in opposite angular directions. To be precise, there is a lift $F=(F_1,F_2):S\to S$  of $f$ such that 
$$
F_1(x,0)<x<F_1(x,1),\quad  x\in\R.\\[1.5ex]
$$

\begin{Thm}\label{thm:PB theorem}
Let $f$ be a reversible map on $A$  isotopic to the identity with the boundary twist condition. Then there is an interior  symmetric fixed point of $f$ on each connected component of $\Fix I_A$. 
\end{Thm}
We point out that here and in the next theorem a reversible map is not even required to have no wandering point, which is a weaker condition than being area-preserving, see \cite{Fra88}. In \cite{Bir15}, Birkhoff already showed the existence of infinitely many symmetric periodic orbits and studied their rotation numbers. There was no precise statement like Theorem \ref{thm:PB theorem} but presumably he knew that this is true. We will outline two proofs of Theorem \ref{thm:PB theorem} which are on the lines of \cite{Bir13,Bir15}.

In fact, we find a necessary and sufficient condition on the existence of a symmetric fixed point of a reversible map isotopic to the identity on an arbitrary invariant connected domain $\Om$ and in particular, the boundary twist condition on $A$ implies the sufficient condition. Note that the fixed locus of $I_\Om$ is composed of disjoint intervals $Y_i$, i.e.  $\Fix I_\Om=\bigsqcup_{1\leq i\leq n} Y_i$.

\begin{Thm}\label{thm:necessary and sufficient condition}
Let $f:\Om\to\Om$ be a reversible map isotopic to the identity. Then $f$ has a  symmetric fixed point on $Y_i$ if and only if $f(Y_i)\cap Y_i\neq\emptyset$ for $i\in\{1,\dots,n\}$ Similarly, the existence of an interior symmetric fixed point of $f$ on $Y_i$ is equivalent to $f(Y_i)\cap Y_i\cap\mathring \Om\neq\emptyset$.  
\end{Thm}

In general it is not easy to find an analogue of the boundary twist condition for an arbitrary possibly non-closed planar domain, see \cite{Fra88} for the case of the open annulus.  However in view of this theorem, a reversible map $f:\Om\to\Om$ is reasonable to be called a  {\bf twist map} if $\Fix I_\Om$ is twisted, explicitly, if some component of $\Fix I_\Om$ has a self-intersection point by $f$. 

Apart from these results, we will also show some classical fixed point theorems for symmetric fixed points of reversible maps and some of them will play key roles in the proof of our main results.

The Poincar\'e-Birkhoff theorem has been proved and generalized in many papers \cite{Bir13,Bir15,Bir25,BN77,Neu77,Car82,Din83,Fra88,YZ95,Gui97,LCW09,HMS12,KS13,PR13}. Franks' theorem also has been pursued further in \cite{FH03,LeC06,CKRTZ12,Ker12}. There has been a remarkable new approach  using pseudoholomorphic curves to understand smooth dynamics on planar domains by Bramham and Hofer \cite{BH12} (see also the literature cited therein). We treat the Poincare-Birkhoff theorem and Franks' theorem in their original forms for brevity but we expect that such generalized works might have reversible counterparts as well.

\subsection{Applications to symmetric periodic orbits}\quad\\[-1.5ex]

Let $M$ be a 3-dimensional smooth manifold with a smooth vector field $X$. An embedded compact Riemann surface $\Sigma\into M$ with boundary is called a {\bf global surface of section} for $X$ if $(\Sigma,X)$ meets the following requirements.\\[-2ex]
\begin{itemize}
\item[(1)] The boundary of $\Sigma$ consists of periodic orbits, called the {\em spanning orbits}.
\item[(2)] The  vector field $X$ is transversal to the interior $\mathring\Sigma$ of $\Sigma$.
\item[(3)] Every orbit of $X$, except the spanning orbits, passes through $\mathring\Sigma$ in forward and backward time.\\[-2ex]
\end{itemize}

When we study dynamical systems with two degrees of freedom, the existence of global surfaces of section reduces the complexity by one dimension. A notion of global surfaces of section was introduced by Poincar\'e \cite{Poi99} and further studied by Birkhoff \cite{Bir17}. In the spirit of Poincar\'e's groundbreaking idea, our results help us to find symmetric periodic orbits of reversible dynamical systems with global surfaces of section invariant under symmetries. 
If $M$ carries a smooth involution $\rho$ with $\rho_*X=-X$, a global surface of section $\Sigma$ invariant under $\rho$, i.e. $\rho(\Sigma)=\Sigma$, is called an {\bf invariant global surface of section}. Note that the unique spanning orbit $\gamma:\R/T\Z\to M$ of an invariant global disk-like surface of section is automatically {\bf symmetric}, i.e. $\rho(\gamma(\R/T\Z))=\gamma(\R/T\Z)$, or equivalently if $\rho(\gamma(t))=\gamma(-t)$, $t\in \R/T\Z$ up to reparametrization.   The condition $\rho^*X=-X$ yields that the induced Poincar\'e return map $f:\mathring\Sigma\to\mathring\Sigma$ is reversible with respect to the involution $\rho|_\Sigma$. A crucial observation is that there is the following one-to-one correspondence:
$$
\left\{\begin{array}{ll}\textrm{geometrically distinct sym-}\\ \textrm{metric periodic orbits of $X$}\end{array}\right\}\setminus
\left\{\begin{array}{ll}\textrm{spanning}\\ \textrm{orbits of $\Sigma$}\end{array}\right\}
 \longleftrightarrow 
\left\{\begin{array}{ll}\textrm{$f$-orbits of symmetric}\\ \textrm{periodic points of $f$}\end{array}\right\}.
$$
Accordingly, reversible dynamical systems with invariant global surfaces of section, for instance the planar restricted three-body problem or the St\"ormer problem, are of interest to us.
As observed in \cite{Bir15,Con63,McG69}, there are invariant global  surfaces of section in the restricted three-body problem sometimes and the Poincar\'e return maps are area preserving. Although they were not interested in whether their global surfaces of section are invariant, one can easily check that these are indeed invariant. Another example is the St\"ormer problem, see  \cite{Sto07,Bra70}. The global surface of section in the St\"ormer problem constructed in \cite{DeV54} is also invariant.

More generally, Frauenfelder and the author \cite{FK14}  prove that if a hypersurface in $\R^4$ bounds a strictly convex domain and is invariant under the involution $\rho:=\mathrm{diag(-1,-1,1,1)}$, there exists an invariant global disk-like surface of section for the standard Reeb vector field and  the Poincar\'e return map is area preserving, which generalizes a pioneering work of Hofer, Zehnder, and Wysocki \cite{HWZ98,HWZ03}.  This shows that a  global disk-like surface of section in  \cite{AFFHvK12} is also invariant. As a consequence of Corollary \ref{cor:dichotomy}, there are either two or infinitely many symmetric periodic orbits on such a hypersurface and the latter happens under the presence of a non-symmetric periodic orbit.
\\[-1.5ex]

Suppose that a hypersurface $M\subset\R^4$ is centrally symmetric, i.e. $-\Id_{\R^4}(M)=M$. The standard Reeb vector field $X$ on $M$ automatically satisfies $(-\Id_{\R^4})^*X=X$. A periodic orbit which is symmetric with respect to $-\Id_{\R^4}$ is called {\bf centrally symmetric}. A centrally symmetric hypersurface $M$ which is symmetric with respect to $\rho$ as well is also interesting as the regularized energy hypersurface of the planar restricted three-body problem is so, see \cite{AFFHvK12}. If a periodic orbit is symmetric with respect to $\rho$ as well as $-\Id_{\R^4}$, it is called  {\bf doubly symmetric}. The geometric motion of doubly symmetric periodic orbits and that of $\rho$-symmetric centrally non-symmetric periodic orbits are considerably different, see \cite{Kan12}. In this situation, if $\Sigma$ is a $\rho$-invariant global surface of section,  so is $-\Id_{\R^4}(\Sigma)$. In consequence, the Poincar\'e (half-return) map $f:\mathring\Sigma\to -\Id_{\R^4}(\mathring\Sigma)=\mathring\Sigma$ (the identification by $-\Id_{\R^4}$) encodes the qualitative properties of the doubly symmetric aspect of $X$, and in particular we have the following correspondence.
$$
\left\{\begin{array}{ll}\textrm{geometrically distinct doubly}\\ \textrm{symmetric periodic orbits of $X$}\end{array}\right\}\setminus\left\{\begin{array}{ll}\textrm{spanning}\\ \textrm{orbits of $\Sigma$}\end{array}\right\}
 \longleftrightarrow 
\left\{\begin{array}{ll}\textrm{$f$-orbits of symmetric}\\ \textrm{odd-periodic points of $f$}\end{array}\right\}.
$$
In view of this correspondence, Theorem \ref{thm:dichotomy2} helps us to detect doubly symmetric periodic orbits. We expect that a hypersurface in $\R^4$ bounding a strictly convex domain and symmetric with respect to both $\rho$ and $-\Id_{\R^4}$ has a $\rho$-invariant global disk-like surface of section with the doubly symmetric spanning orbit.\\[-1.5ex]

We can apply this to symmetric closed geodesics on Riemannian 2-spheres. A combination of Franks' work  \cite{Fra92} and Bangert's work \cite{Ban93} (see also \cite{Hin93}) proves the existence of infinitely many closed geodesics on every Riemannian 2-sphere. In particular Franks' work applies to a Riemannian 2-sphere with Birkhoff's global annulus-like surface of section. Let $\gamma:S^1\to S^2$ be a simple closed geodesic with respect to a Riemannian metric $g$. It separates $S^2$ into two disks and we denote one of them by $D_1$. Then 
$$
\Sigma_\gamma:=\{(x,v)\in S_gS^2\,|\, x\in\gamma(S^1),\; 0\leq \angle(\dot\gamma_x,v)\leq\pi\}
$$
is an embedded surface in the unit tangent bundle $S_gS^2$ diffeomorphic to the closed annulus with boundary $\dot\gamma(S^1)\sqcup\dot{\bar\gamma}(S^1)$ where $\bar\gamma(t):=\gamma(-t)$. Birkhoff showed in \cite{Bir27} that $\Sigma_\gamma$ is a global surface of section and the Poincar\'e return map on it extends to the boundary of $\Sigma_\gamma$ if $\gamma$ meets certain properties, which are true for  $(S^2,g)$ with positive curvature. Moreover the extended Poincar\'e return map on ${\Sigma}_\gamma$ is area preserving and isotopic to the identity.

There is a refinement of this in the symmetric case as follows. Suppose that there exists an involutive orientation reversing isometry $R$ on $(S^2,g)$, i.e. $R^2=\Id_{S^2}$, $\Fix R\cong S^1$, and $R^*g=g$. We call a closed geodesic $\gamma$ is {\bf symmetric} if $R(\gamma(t))=\gamma(-t)$.
Assume that  there is a simple symmetric closed geodesic $\gamma$  which defines
 a global surface of section $\Sigma_\gamma$. Then $\Sigma_\gamma$ is invariant under the involution $R_*$. Hence the following corollary is an immediate consequence of Theorem \ref{thm:Franks-like theorem} and Theorem \ref{thm:PB theorem}, see the proof of \cite[Theorem 4.1]{Fra92} for details.

\begin{Cor}\label{cor:geodesic}
There are infinitely many symmetric closed geodesics on a symmetric $(S^2,g,R)$ equipped with invariant Birkhoff's global annulus-like surface of section.
\end{Cor}

\subsection{Questions}\quad\\[-1.5ex]

We expect that an analogue of Theorem \ref{thm:dichotomy2} in the absence of reversibility is also true. To be precise, we expect that every area preserving homeomorphism $f$ on $A$ or $\mathring A$ isotopic to the identity has infinitely many odd-periodic points  provided a single odd-periodic point. More generally it is conceivable that the following claim is true. Assume that an area preserving homeomorphism $f$ on $A$ or $\mathring A$ isotopic to the identity has a $k$-periodic point for some $k\in\N$. For a given $n\in\N$ if $n$ and $k$ are relatively prime, $f$ has infinitely many periodic orbits with period relatively prime to $n$. This is motivated by the case $S^3$ with the standard $\Z/n$-action. Note that this covers the centrally symmetric case when $n=2$.

It is also tempting to remove the assumption on the existence of Birkhoff's global surface of section in Corollary \ref{cor:geodesic} by showing a reversible counterpart to a result in \cite{Ban93}.


\section{Reversible maps and involutions}

A remarkable property of reversible maps is that $f\circ I_\Omega$ is an involution if $f$ is a reversible map on an invariant domain $\Omega$. Symmetric fixed points of a reversible maps $f$ can be interpreted as intersection points of two fixed loci $\Fix f$ and $\Fix I_\Omega$. Therefore to study a symmetric fixed point of $f$, it is crucial to observe the fixed locus $\Fix (f\circ I_\Omega)$ of the involution $f\circ I_\Omega$ since 
$$
\Fix f\cap \Fix I_\Omega=\Fix (f\circ I_\Omega)\cap \Fix I_\Omega.
$$
\begin{Rmk}
We define an operation $\star$ on the space $\mathfrak{Inv}$ of involutions on $\Omega$ by
$$
J\star K=J\circ K\circ J,\quad J,\,K\in\mathfrak{Inv},
$$
so that $(\mathfrak{Inv},\star)$ becomes a magma, i.e. $J\star K\in\mathfrak{Inv}$. We also consider the space $\mathfrak{Rev}$ of reversible maps on $\Omega$ and endow $\mathfrak{Rev}$ with a magma operation by
$$
f\diamond g=f\circ g^{-1}\circ f,\quad f,\,g\in\mathfrak{Rev}.
$$
Then we have a bijective magma homomorphism between them:
$$
(\mathfrak{Rev},\diamond)\to(\mathfrak{Inv},\star),\quad f\mapsto f\circ I_\Omega.
$$
In particular, any reversible map $f$ is a composition of two involutions $f\circ I_\Omega$ and $I_\Omega$ as observed by Birkhoff \cite{Bir15}.
\end{Rmk}

Recall that a map $f:\Omega\to\Omega$ is said to be {\bf conjugate} to a map $g:\Omega\to\Omega$ if there exists a homeomorphism $h:\Omega\to\Omega$ with $h\circ f=g\circ h$. If this is the case, 
$$
h(\Fix f)=\Fix g.
$$
According to Brouwer \cite{Brou19}, any involution on $\R^2$ is conjugate to $\Id$ or $-\Id$ or $I$. Consequently, the fixed locus of an orientation reversing involution on $\R^2$ is an embedded line and the fixed locus of an involution on a planar domain is a topological submanifold. The following lemma immediately follows from this observation.

\begin{Lemma}\label{lem:fixd loci of involutions}
If $J$ is an orientation reversing involution on $\Omega$, $\Fix J$ is a topological submanifold. If $\Omega=\R^2$, $\Fix J\cong \R$. If $\Omega =D$, $\Fix J\cong [0,1]$.  If $\Omega=S$ and $J$ is isotopic to $I_S$, $\Fix J\cong[0,1]$. If $\Omega=A$ and $J$ is isotopic to $I_A$, $\Fix J\cong[0,1]\sqcup[0,1]$.
\end{Lemma}

\begin{Rmk}\label{rmk:Lefschetz fixed point theorem}
In general it is not easy to answer whether an involution has a fixed point. There is a Lefschetz fixed point theorem for involutions due to \cite{KK68}. This states that  if an involution $J$ on a locally compact space $X$ is fixed point free,
$$
\Lambda_J:=\sum_{k\geq0}(-1)^k\mathrm{Tr} (J_*|H_k(X;\R))=0.
$$
Even if $\Fix J$ is nonempty, it is not necessarily a topological submanifold whereas the fixed locus of a smooth involution is always a smooth submanifold. There is a so-called {\em dogbone space} $B$ due to Bing which is not a topological manifold but $B\x\R$ is homeomorphic to $\R^4$. Using this, one can construct a continuous involution $J$ on $\R^4$ with $\Fix J=B$.
We refer the reader to  \cite[Section 9]{Bin59} for details.
\end{Rmk}

\begin{Rmk}
The condition $J\simeq I_A$ when $\Omega=A$ in the preceding lemma is indispensable. Note that an antisymplectic involution $J$ on $A$ defined by
$$
(r,\theta)\mapsto (3-r,\pi+\theta)
$$
is fixed point free where $(r,\theta)\in\R^2$ is the polar coordinate. In particular, the Lefschetz number $\Lambda_{J}$ in Remark \ref{rmk:Lefschetz fixed point theorem} vanishes.
\end{Rmk}

The following two easy lemmas will play key roles. Note that any iteration $f^i$, $i\in\Z$ of a reversible map $f$ is reversible again.
\begin{Lemma}\label{lem:fix loci}
If $f:\Omega\to\Omega$ is reversible, 
$$
f^j(\Fix (f^i\circ I_\Omega))=\Fix(f^{2j+i}\circ I_\Omega), \quad\forall i,\, j\in\Z.
$$
\end{Lemma}
\begin{proof}
For $z\in\Fix (f^i\circ I_\Omega)$, $f^j(z)\in\Fix(f^{2j+i}\circ I_\Omega)$ since
$$
f^{2j+i}\circ I_\Omega(f^j(z))=f^{2j+i}\circ f^{-j}(I_\Om(z))=f^j(f^i\circ I_\Omega(z))=f^j(z).
$$
To prove the converse we pick $z\in\Fix(f^{2j+i}\circ I_\Omega)$. Then,
$$
f^i\circ I_\Omega(f^{-j}(z))=I_\Omega\circ f^{-i-j}(f^{2j+i}\circ I_\Omega(z))=I_\Omega\circ f^j\circ I_\Omega(z)=f^{-j}(z).
$$
This implies $f^{-j}(z)\in\Fix(f^i\circ I_\Omega)$ and thus $z\in f^j(\Fix f^i\circ I_\Omega)$.
\end{proof}

\begin{Lemma}\label{lem:intersection point=>symmetric periodic point}
Let $f:\Omega\to\Omega$ be a reversible map. If $z\in\Fix (f^k\circ I_\Omega)\cap\Fix(f^\ell\circ I_\Omega)$ for some $k>\ell\in\Z$, then $z$ is a symmetric $(k-\ell)$-periodic point of $f$.
\end{Lemma}
\begin{proof}
Since $f^k\circ I_\Omega(z)=f^\ell\circ I_\Omega(z)$, 
$$
z=I_\Omega\circ f^{\ell-k}\circ I_\Omega(z)=f^{k-\ell}(z).
$$
Moreover by the assumption $f^k\circ I_\Omega(z)=z$ and hence
$$
f^{m(k-\ell)-k}(z)=f^{-k}(z)=I_\Omega(z).
$$
for any $m\in\Z$.
\end{proof}

\section{On the fixed point index for reversible maps}

For any two smooth paths $\delta,\,\gamma:[0,1]\to\Omega$ with $\delta(t)\neq\gamma(t)$ for every $t\in[0,1]$, we define a path $\gamma-\delta:[0,1]\to\Omega$ avoiding the origin  by $(\gamma-\delta)(t):=\gamma(t)-\delta(t)$. The variation of angle of such two paths $\delta,\,\gamma$ is defined by
$$
i(\delta,\gamma):=\frac{1}{2\pi}\int_{\gamma-\delta} d\theta
$$
where $\theta$ is the angular coordinate on $\R^2$. We list some useful properties of $i(\delta,\gamma)$.

\begin{itemize}
\item[1)] For any $\delta,\,\gamma:[0,1]\to\Om$ with $\delta(t)\neq\gamma(t)$,
$$
i(\delta,\gamma)=i(\gamma,\delta).
$$ 
\item [2)] If $\delta_1,\,\delta_2,\,\gamma_1,\,\gamma_2:[0,1]\to\Omega$ with $\delta_j(t)\neq\gamma_j(t)$, $j=1,2$ satisfy $\delta_1(1)=\delta_2(0)$ and $\gamma_1(1)=\gamma_2(0)$,
$$
i(\delta_1*\delta_2,\gamma_1*\gamma_2)=i(\delta_1,\gamma_1)+i(\delta_2,\gamma_2)
$$
where $*$ stands for the catenation operation defined by 
$$
\gamma_1*\gamma_2(t):=\left\{\begin{array}{ll} \gamma_1(2t) & t\in[0,\frac{1}{2}],\\[.5ex]
\gamma_2(1-2t) & t\in[\frac{1}{2},1].
\end{array}\right.
$$
\item [2)] If $\delta,\,\gamma:[0,1]\to\Om$ with $\delta(t)\neq\gamma(t)$ are loops, i.e. $\delta(0)=\delta(1)$, $\gamma(0)=\gamma(1)$,
$$
i(\delta,\gamma)\in\Z.
$$
\item [3)] For any $\delta,\,\gamma:[0,1]\to\Om$ with $\delta(t)\neq\gamma(t)$,  
$$
i(\bar\delta,\bar\gamma)=-i(\delta,\gamma)
$$
where  $\bar\delta(t):=\delta(1-t)$, $\bar\gamma(t):=\gamma(1-t)$.
\item [4)] Recall that we have assumed that $\Omega$ is invariant.  For any $\delta,\,\gamma:[0,1]\to\Om$ with $\delta(t)\neq\gamma(t)$,  
$$
i(I_\Omega\circ\delta,I_\Omega\circ\gamma)=-i(\delta,\gamma).
$$
\item [5)] For continuous families of paths $\delta_s$,\,$\gamma_s:[0,1]\to\Om$, $s\in\R$ with $\delta_s(t)\neq\gamma_s(t)$ for any $s$, $i(\delta_s,\gamma_s)$ varies continuously with the parameter $s$.
\end{itemize}

\begin{Prop}\label{prop:mirror loop has the same index}
Let $f:\Omega\to\Omega$ be a reversible map isotopic to the identity. Then for any loop $\gamma:S^1\to\Omega\setminus\Fix f$,
$$
i(\gamma,f\circ\gamma)=i(I_\Om\circ\bar\gamma,f\circ I_\Om\circ\bar\gamma).
$$
Let $h:\Omega\to\Omega$ be a reversible map isotopic to $I_\Om$. Then for any loop $\gamma:S^1\to\Omega\setminus\Fix h$,
$$
i(\gamma,h\circ\gamma)=-i(I_\Om\circ\bar\gamma,h\circ I_\Om\circ\bar\gamma).
$$
\end{Prop}
\begin{proof}
Observe that by reversibility of $f$
\bean
i(I_\Om\circ\bar\gamma,f\circ I_\Om\circ\bar\gamma)&=i(I_\Om\circ\bar\gamma,I_\Om\circ f^{-1}\circ\bar\gamma)
=i(\gamma,f^{-1}\circ\gamma).
\eea
Let $f_s:\Omega\to\Omega$, $s\in[0,1]$ be an isotopy between $f_0=\Id_\Om$ and $f_1=f$. Since $i(f_s\circ\gamma,f_s\circ f^{-1}\circ\gamma)$ varies continuously in $s\in[0,1]$ and it takes an integer value for every $s\in[0,1]$, $i(f_s\circ\gamma,f_s\circ f^{-1}\circ\gamma)$ is constant. In particular,
$$
i(\gamma,f^{-1}\circ\gamma)=i(f\circ\gamma,\gamma),
$$
and the first assertion is proved. For the second identity, let $h_s:\Omega\to\Omega$, $s\in[0,1]$ be an isotopy between $h_0=I_\Om$ and $h_1=h$. In a similar vein, we compute
\bean
i(I_\Om\circ\bar\gamma,h\circ I_\Om\circ\bar\gamma)&=i(\gamma,h^{-1}\circ\gamma)
=-i(I_\Om\circ\gamma,I_\Om\circ h^{-1}\circ\gamma)=-i(h\circ\gamma,\gamma).
\eea
\end{proof}

Recall that the \textbf{index} of an isolated interior fixed point $z\in\Fix f$ is defined by
$$
i(f,z):=i(\gamma,f\circ\gamma)
$$
for any simple sufficiently small loop $\gamma:S^1\to\Omega\setminus\Fix f$ winding $z$ counterclockwise. We will also use the 
fact that for any simply connected domain $D_\delta$ enclosed by a simple loop $\delta:S^1\to\Omega\setminus\Fix f$ with finitely many fixed points of $f$, it holds that
$$
i(\delta,f\circ\delta)=\sum_{x\in\Fix f\cap D_\delta} i(f,x).
$$

\begin{Cor}\label{cor:index}
Let $f:\Omega\to\Omega$ be a reversible map. If $z$ is fixed by $f$, so is $I_\Omega(z)$. Moreover if $f$ is isotopic to the identity and $z\in\Fix f\cap \mathring\Om$, $i(f,z)=i(f,I_\Om(z))$. If $f$ and $I_\Om$ are isotopic, $i(f,z)=-i(f,I_\Om(z))$
\end{Cor}
\begin{proof}
If $f(z)=z$, $f\circ I_\Om(z)=I_\Om\circ f^{-1}(z)=I_\Om(z)$. Assume that $f$ is isotopic to the identity. The case that $f$ is isotopic to $I_\Om$ follows in a similar way. Suppose that $\gamma:S^1\to\Omega\setminus\Fix f$ is a sufficiently small loop such that $i(f,z)=i(\gamma,f\circ\gamma)$ for $z\in\Fix f$. Then $I_\Om\circ\bar\gamma$ is a sufficiently small loop surrounding $I_\Om(x)$ counterclockwise, and hence using Proposition \ref{prop:mirror loop has the same index} we have
$$
i(f,I_\Om(z))=i(I_\Om\circ\bar\gamma,f\circ I_\Om\circ\bar\gamma)=i(\gamma,f\circ\gamma)=i(f,z).
$$
\end{proof}

\section{Classical fixed point theorems for reversible maps}
 The following statement can be thought as Brouwer's fixed point theorem for symmetric fixed points of reversible maps.  Note that any continuous map from $D$ to itself is isotopic to $\Id_{D}$ (to $I_D$) provided that it is orientation preserving (reversing).

\begin{Prop}
Let $f:D\to D$ be a reversible map with finitely many fixed points. Then $f$ has at least one symmetric fixed point.
\end{Prop}
\begin{proof}
Suppose that there is no symmetric fixed point of $f$. Then due to Corollary \ref{cor:index}, 
$$
\sum_{z\in\Fix f}i(f,z)\in2\Z.
$$
This contradicts the Lefschetz-Hopf theorem, which implies
$\sum_{z\in\Fix f}i(f,z)=1$.
\end{proof}

\begin{Rmk}
There is a reversible map $f:D\to D$ with infinitely many fixed points but no symmetric fixed point.  For simplicity we consider the domain $B=\{(x,y)\in\R^2\,|\, -2\leq x\leq 2,\,0\leq y\leq 1\}$ instead of $D$. Let $f:B\to B$ be a reversible map given by 
$$
\left\{\begin{array}{cc} f(x,y)=(4+3x,y),&\quad y\in[-2,-1]\\[.5ex]
f(x,y)=\Big(\frac{4}{3}+\frac{1}{3}x,y), &y\in[-1,2]
\end{array}\right.
$$
Then $\Fix f=\{(-2,y)\cup(2,y)\,|\,0\leq y\leq 1\}$ but $f$ has no symmetric periodic point.
\end{Rmk}

The following is a version of Brouwer's lemma \cite{Brou12} for reversible maps.
\begin{Prop}
Let $f:\R^2\to\R^2$ be an orientation preserving reversible map with a symmetric periodic point. Then $f$ has a symmetric fixed point.
\end{Prop}
\begin{proof}
Let $(x,y)\in\R^2$ be a symmetric periodic point of $f$, i,e, $f^\ell(x,y)=I(x,y)=(-x,y)$ and $f^k(x,y)=(x,y)$ for some $\ell,\, k\in\N$. Since $f\circ I$ is an orientation reversing involution, $\Fix(f\circ I)$ is an embedded line in the plane due to Lemma \ref{lem:fixd loci of involutions}. We write $\alpha=\Fix(f\circ I)$. Assume by contradiction that $\alpha\cap\Fix I=\emptyset$, see Lemma \ref{lem:intersection point=>symmetric periodic point}. Note that $I(\alpha)$ is also an embedded line which does not intersect with both $\Fix I$ and $\alpha$. We assume that $I(\alpha)\subset \{(x,y)\in\R^2\,|\, x<0\}$. The case that $I(\alpha)\subset \{(x,y)\in\R^2\,|\, x>0\}$ also follows in a similar way.  The line $\alpha$ (resp. $I(\alpha)$) divides the plane into two open regions $\alpha_+$ and $\alpha_-$ (resp. $I(\alpha)_+$ and $I(\alpha)_-$). Let $\alpha_-$ and $I(\alpha)_+$ be regions containing the imaginary axis $\Fix I $.  Since $f$ is orientation preserving and $f(I(\alpha))=\alpha$, $f$ maps $I(\alpha)_\pm$ to $\alpha_\pm$ respectively. Now we show that $(x,y)$ cannot be a symmetric periodic point. If $(x,y)\in I(\alpha)_-$, $f^n(x,y)\in\alpha_-$ for all $n\in\N$ but $f^\ell(x,y)=(-x,y)\in\alpha_+$ for some $\ell\in\N$. Thus this does not happen and othere cases $(x,y)\in I(\alpha)\cup I(\alpha)_+$ can be ruled out in a similar way. This contradiction proves the proposition.
\end{proof}

The above proposition does not remain true for general invariant domains. However we still be able to find a symmetric fixed point in the presence of a certain symmetric 2-periodic point. Note that if $(0,y)\in\Fix I_\Om$ is a symmetric 2-periodic point of a reversible map $f:\Om\to\Om$, $f(0,y)\in\Fix I_\Om$ as well. Borrowing an idea of M. Brown \cite{Bro90}, if a symmetric 2-periodic points $(0,y),\,f(0,y)\in\Fix I_\Om$ lie on the same connected component of $\Fix I_\Om$, then we can find a symmetric fixed point of a reversible map $f$ isotopic to the identity.  This will be used later to prove Theorem \ref{thm:necessary and sufficient condition}. 

\begin{Prop}\label{period 2 real periodic point=>real fixed point}
Suppose that a  reversible map $f:\Om\to\Om$ isotopic to the identity has a symmetric 2-periodic point $(0,y)\in\Fix I_\Om$. If $(0,y)$ and $f(0,y)$ are on the same connected component of $\Fix I_\Om$, there exists a symmetric fixed point of $f$ lying on $[(0,y),f(0,y)]\subset\Fix I_\Om$.
\end{Prop}
\begin{proof}
We may assume that $f$ has no fixed point in a small open neighborhood of the interval $[(0,y),f(0,y)]$ since otherwise $f$ has a fixed point on $[(0,y),f(0,y)]$ which is in turn a symmetric fixed point.
We choose an arc $\gamma:[0,1]\to\Om\setminus(\Fix f\cup\Fix I_\Om)$ with $\gamma(0)=(0,y)$ and $\gamma(1)=f(0,y)$ such that the invariant domain $D_\gamma$ enclosed by the loop $\gamma*(I_\Om\circ\bar\gamma)$ is simply connected and contains no fixed points of $f$ . 

Since $f\circ\gamma(1)-\gamma(1)=(0,y)-f(0,y)$ and $f\circ\gamma(0)-\gamma(0)=f(0,y)-(0,y)$ point in opposite directions,
$$
i(\gamma,f\circ\gamma)\in\frac{1}{2}+\Z.
$$
Now we claim that 
$$
i(I_\Om\circ\bar\gamma,f\circ I_\Om\circ\bar\gamma)= i(\gamma,f\circ\gamma)
$$
where $\bar\gamma(t)=\gamma(1-t)$. Indeed,
$$
i (I_\Om\circ\bar\gamma,f\circ I_\Om\circ\bar\gamma)=-i(I_\Om\circ\gamma,f\circ I_\Om\circ\gamma)=-i(I_\Om\circ\gamma,I_\Om\circ f^{-1}\circ\gamma)=i(\gamma,f^{-1}\circ\gamma).
$$
Let $f_s:\Om\to\Om$, $s\in[0,1]$ be an isotopy between $f_0=\Id$ and $f_1=f$. Since $f_s\circ f^{-1}\circ\gamma(1)-f_s\circ\gamma(1)=f_s(0,y)-f_s\circ f(0,y)$ and $f_s\circ f^{-1}\circ\gamma(0)-f_s\circ\gamma(0)=f_s\circ f(0,y)-f_s(0,y)$  point in opposite directions, 
$$
i(f_s\circ\gamma,f_s\circ f^{-1}\circ\gamma)\in\frac{1}{2}+\Z.,
$$ 
for every $s\in[0,1]$ and in particular it holds that
$$
i(\gamma,f^{-1}\circ\gamma)=i(f\circ \gamma,\gamma).
$$
This proves the clam. By the claim, we have
$$
i(\gamma*(I_\Om\circ\bar\gamma),f(\gamma*(I_\Om\circ\bar\gamma)))=i(\gamma,f\circ\gamma)+i(I_\Om\circ\bar\gamma,f\circ I_\Om\circ\bar\gamma)=2\,i (\gamma,f\circ\gamma). 
$$
Using $i(\gamma,f\circ\gamma)\in\frac{1}{2}+\Z$, we deduce
$$
i(\gamma*(I_\Om\circ\bar\gamma),f(\gamma*(I_\Om\circ\bar\gamma)))\in 1+2\Z,
$$
and therefore there is a fixed point of $f$ in $D_\gamma$. This contradiction shows that there exists at least one symmetric fixed point of $f$ inside $D_\gamma$, and hence on $\Fix I_\Om\cap D_\gamma=[(0,y),f(0,y)]$.
\end{proof}


\section{Proofs of the main results}
\subsection{Proofs of Theorem \ref{thm:Franks-like theorem} and Corollary \ref{cor:dichotomy}}

\begin{Prop}\label{prop:per pt=>int sym per pt}
If an area preserving reversible map $f$ on $A$ or $\mathring A$ has a periodic point, there is an interior symmetric periodic point of $f$.
\end{Prop}
\begin{proof}
\begin{figure}[htb]
\includegraphics[width=1.0\textwidth,clip]{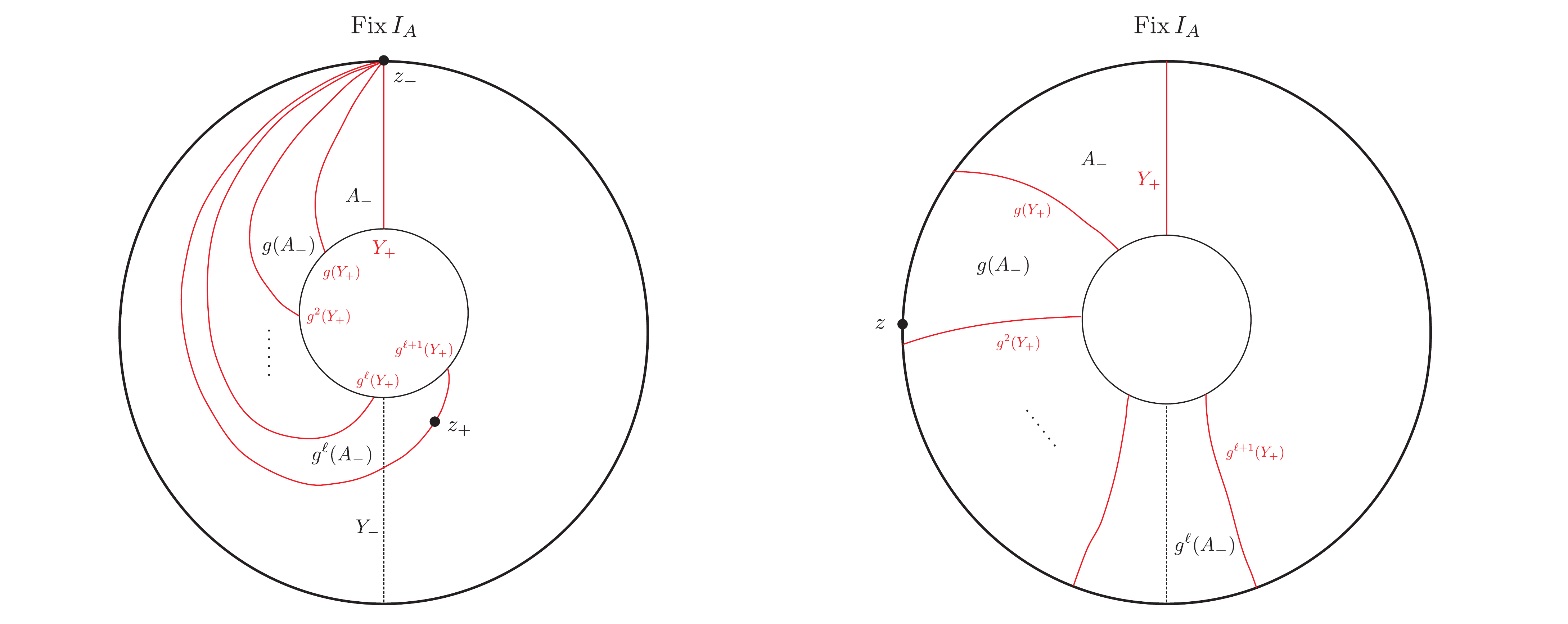}
\caption{}\label{fig:disk0}
\end{figure}
We only give a proof for $A$ which applies to $\mathring A$ as well. Assume by contradiction that there is no interior symmetric periodic point of $f$. We denote the connected components of $\Fix I_A$ by $Y_\pm:=\{0\}\x[\pm1,\pm2]$. By Lemma \ref{lem:fix loci}, $f^m(Y_+)\subset\Fix(f^{2m}\circ I_A)$ for every $m\in\N$, and thus by Lemma \ref{lem:intersection point=>symmetric periodic point}, we have
\beq\label{eq:non-intersections}
f^m(Y_+)\cap Y_\pm\cap\mathring A=\emptyset,\quad m\in\N\cup\{0\}.
\eeq
Two segments  $Y_+$ and $f(Y_+)$ divide $A$ into two closed regions $A_-$ and $A_+$ such that $A_-\cup A_+=A$ and $A_-\cap A_+=Y_+\cup f(Y_+)$. Since $f(Y_+)\cap Y_-\cap\mathring A=\emptyset$, $A_-$ and $A_+$ have different area, say $\mathrm{area}(A_-)<\mathrm{area}(A_+)$. We assume that $A_-\subset\{(x,y)\in\R^2\,|\,x\leq 0\}$. The other case is proved in the same manner. Note that 
\beq\label{eq:area preserving}
\mathrm{area}(A_-)=\mathrm{area}(f^m(A_-))<\frac{\mathrm{area}(A)}{2},\quad m\in\N\cup\{0\}.
\eeq
 We claim that  
\beq\label{eq:boundary intersections}
f^m(A_-)\cap f^{m+1}(A_-)=f^m(Y_+),\quad m\in\N\cup\{0\}.
\eeq 
It suffices to prove the claim for $m=0$. Indeed if this is not true, $A_-\cap f(A_-)\supsetneq f(Y_+)$, and thus $f^2(Y_+)\cap Y_+\cap\mathring A\neq\emptyset$ since $\mathrm{area}(A_-)+\mathrm{area}(f(A_-))<\mathrm{area}(A)$. This contradicts \eqref{eq:non-intersections} and hence \eqref{eq:boundary intersections} follows.

Now we are ready to prove the assertion.
Suppose that $z_-\in Y_+\cap f(Y_+)$. By \eqref{eq:non-intersections}, $z_-\in\p A$. Observe that $f^\ell(A_-)\cap\{(x,y)\in A\,|\, x>0\}\neq\emptyset$ for some $\ell\in\N$ by \eqref{eq:area preserving} and \eqref{eq:boundary intersections}. If $\ell\in\N$ is the minimal number with this property, there exists a point 
$$
z_+\in f^{\ell+1}(Y_+)\cap\{(x,y)\in A\,|\, x>0\}.
$$ 
Since $z_-$ is fixed by $f$,  $f^{\ell+1}(Y_+)$ connects $z_-$ and $z_+$, and therefore $f^{\ell+1}(Y_+)\cap Y_-\cap\mathring A\neq\emptyset$ which ensures the existence of an interior symmetric periodic point of $f$ by Lemma \ref{lem:intersection point=>symmetric periodic point}. See the first picture of Figure \ref{fig:disk0}.

Suppose that $Y_+\cap f(Y_+)=\emptyset$. Let $z\in A$ be a periodic point of $f$, i.e. $f^k(z)=z$ for some $k\in\N$. Abbreviate by $g=f^{k}$. Then 
$$
z\notin g^m(Y_+),\quad m\in\N\cup\{0\}
$$ 
since otherwise $z=g(z)\in Y_+\cap g(Y_+)$.
 We may assume that $z\in\{(x,y)\in A\,|\,x\leq 0\}$ as $I_A(z)$ is also a periodic point of $f$. As before $Y_+$ and $g(Y_+)$ divides $A$ two closed regions and we call the smaller region $A_-$ again. With $g$ and this $A_-$, \eqref{eq:non-intersections}, \eqref{eq:area preserving}, and \eqref{eq:boundary intersections} still hold. By \eqref{eq:area preserving} and \eqref{eq:boundary intersections},
$$
\{(x,y)\in A\,|\,x\leq 0\}\subset \bigcup_{0\leq m\leq \ell}g^m(A_-).
$$ 
Hence the periodic point $z\in g^m( A_-)$ for some $0\leq m\leq \ell$. This contradicts that $z\in\Fix g$ since $g^m(A_-)\cap g^{m+1}(A_-)=g^{m+1}(Y_+)$ by \eqref{eq:boundary intersections} and $z\notin g^{m+1}(Y_+)$, $m\in\N\cup\{0\}$. See the second picture of Figure \ref{fig:disk0}.
\end{proof}

\begin{Thm}\label{thm:sym per pt=>infinite}
If an area preserving reversible map $f$ on $A$ or $\mathring A$ has a periodic point, it possesses infinitely many interior symmetric periodic points.
\end{Thm}
\begin{proof}

\begin{figure}[htb]
\includegraphics[width=0.6\textwidth,clip]{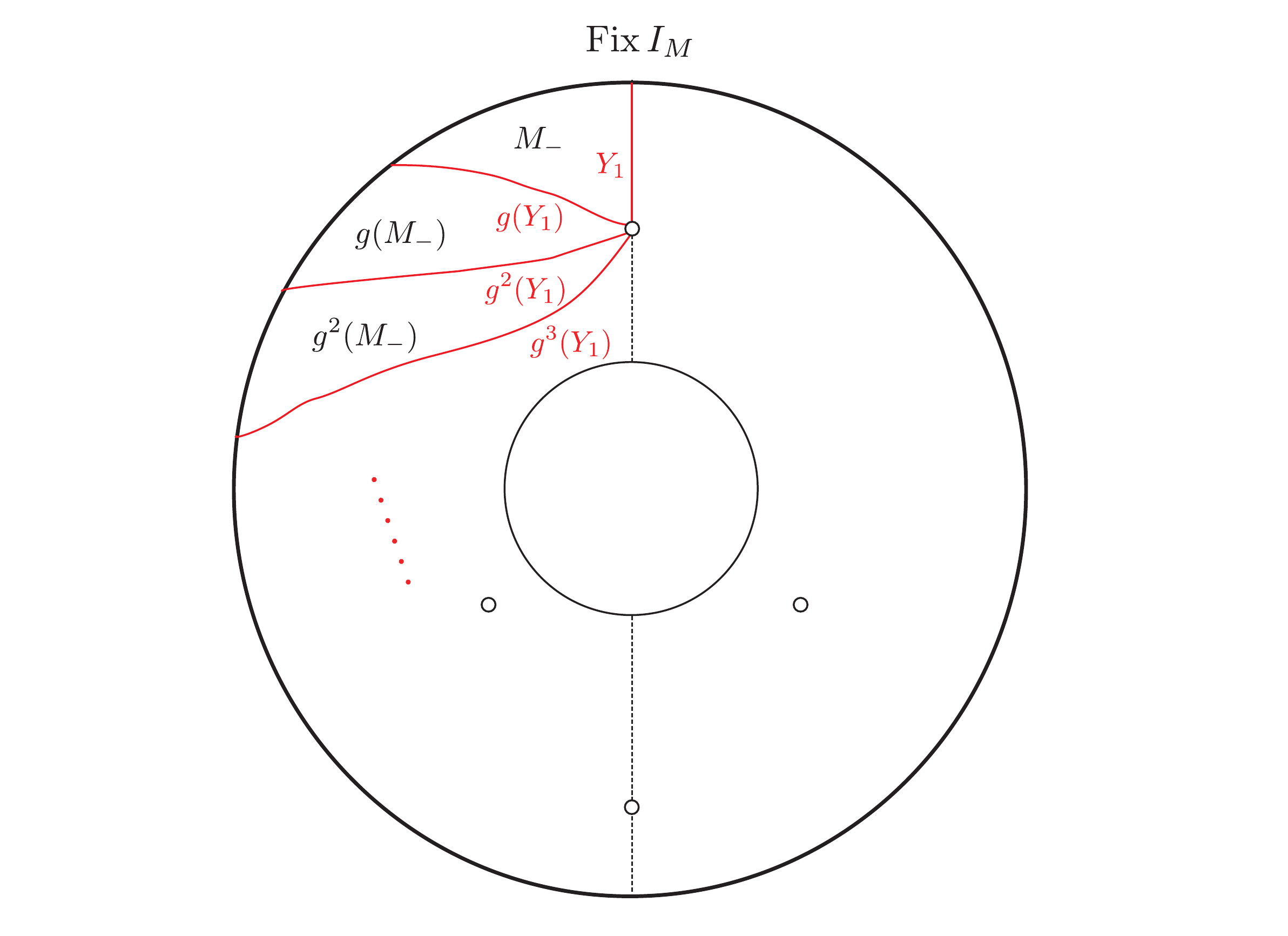}
\caption{}\label{fig:disk1}
\end{figure}

See Figure \ref{fig:disk1}. We will prove the theorem only for $A$ and the same proof goes through for $\mathring A$ as well. Due to Proposition \ref{prop:per pt=>int sym per pt}, there is an interior symmetric periodic point of $f$.
Suppose that there are only finitely many interior symmetric periodic points of $f$. By deleting interior symmetric periodic points from $A$, we obtain a punctured annulus which we denote by $M$. Suppose that $M$ has $n\geq1$ punctures. Since $f$ is homeomorphism, $f$ is permutes $n$ punctures and thus the iterated map $g:=f^{2n!}$ fixes punctures pointwise. Moreover $g$ is isotopic to the identity and in particular $g$ preserves the inner/outer boundary circle of $M$.

The fixed locus $\Fix I_M$ consists of several connected components. We denote by $Y_1$ the component at the very top, i.e. $Y_1$ connects the outer boundary of $A$ and the very top puncture or the inner boundary circle of $A$. We observe that
\beq\label{eq:e_i}
g^m(Y_1)\cap\Fix I_M\cap\mathring M=\emptyset,\quad  m\in\N\cup\{0\}
\eeq
as in \eqref{eq:non-intersections} due to Lemma \ref{lem:fix loci} and Lemma \ref{lem:intersection point=>symmetric periodic point}. 
Note that $Y_1$ and $g(Y_1)$ separates $M$ into two closed regions $M_-$ and $M_+$. Due to \eqref{eq:e_i} with  $m=1$, $M_-$ and $M_+$ have different area, say $\mathrm{area}(M_-)<\mathrm{area}(M_+)$. Observe as before, see  \eqref{eq:boundary intersections}, that 
\beq\label{eq:bdry intersection}
g^m(M_-)\cap g^{m+1}(M_-)=g^m(Y_1),\quad  m\in\N\cup\{0\}.
\eeq 
Since $g$ fixes punctures pointwise, there is no puncture inside every $g^m(M_-)$, $m\in\N\cup\{0\}$ by \eqref{eq:bdry intersection}. There is no loss of generality in assuming $M_-\subset\{(x,y)\in M\,|\,x\leq 0\}$. Due to \eqref{eq:e_i} and the fact that there is at least one pucture on $\{(x,y)\in A\,|\,x\leq 0\}$, every $g^m(Y_1)\subset \{(x,y)\in M\,|\,x\leq 0\}$ for all $m\in\N\cup\{0\}$ which contradicts that $g$ is preserves the area of $M_-$. This completes the proof.
\end{proof}
One can use a covering picture to prove the above theorem, see Theorem \ref{thm:infinitely many odd}.

\begin{Prop}\label{prop:symmetric fixed point}
Every area preserving reversible map $f$ on $D$ or $\mathring D$ has an interior symmetric fixed point.
\end{Prop}
\begin{proof}

The proof is stated for $D$ but the same argument works for $\mathring D$. As observed it suffices to show that 
$$
\Fix (f\circ I_D)\cap\Fix I_D\cap\mathring D\neq\emptyset.
$$
Suppose that $f$ is orientation preserving. Then since $f\circ I_D$ is an orientation reversing area preserving involution, $\Fix(f\circ I_D)$ is an embedded line in $D$ (see Lemma \ref{lem:fixd loci of involutions}) which divides $D$ into two regions with the same area. Thus it has to cross the interior part of $\Fix I_D$.

On the other hand, if $f$ is orientation reversing, we consider $f^2$ instead. Then by the above argument, $f^2$ has an interior symmetric fixed point which in turn implies that $f$ has a symmetric 2-periodic point on $\Fix I_D\cap \mathring D$. Applying Proposition \ref{period 2 real periodic point=>real fixed point}, we can find an interior symmetric fixed point of $f$.
\end{proof}

\subsection{Proof of  Theorem \ref{thm:dichotomy2}}

\begin{Prop}
If an area preserving reversible map $f$ on $A$ or $\mathring A$ isotopic to the identity has an odd-periodic point, there is an interior symmetric odd-periodic point of $f$.
\end{Prop}
\begin{proof}
This follows from almost the same argument as in the proof of Proposition \ref{prop:per pt=>int sym per pt}. Instead repeating the argument, we briefly explain why the proof carries over to this proposition. The first key point is that $f$ is additionally assumed to be isotopic to the identity. This ensures that $\Fix (f\circ I_A)$ separates $A$ into two regions by Lemma \ref{lem:fixd loci of involutions}. The second reason is that points in 
$$
f^m(\Fix(f\circ I_A))\cap \Fix I_A,\quad m\in\N\cup\{0\}
$$ 
are symmetric odd-periodic points according to Lemma \ref{lem:fix loci} and Lemma \ref{lem:intersection point=>symmetric periodic point}. Hence the proof of Proposition \ref{prop:per pt=>int sym per pt} with $Y_\pm$ replaced by two connected components of $\Fix (f\circ I_A)$ proves the present proposition.
\end{proof}

It is conceivable that the following theorem also immediately follows from the proof of Theorem \ref{thm:sym per pt=>infinite} with minor modifications. However we provide a slightly different picture.

\begin{Thm}\label{thm:infinitely many odd}
Let $f$ be an area preserving reversible map on $A$ or $\mathring A$ isotopic to the identity. If $f$ has a odd-periodic point, there are infinitely many interior symmetric odd-periodic points.
\end{Thm}
\begin{proof}
\begin{figure}[htb]
\includegraphics[width=1.\textwidth,clip]{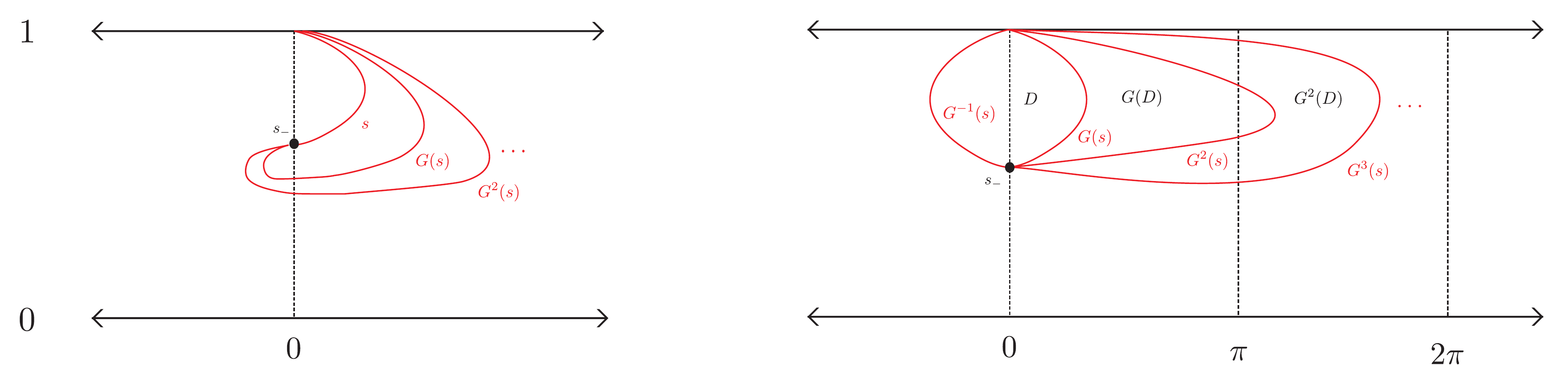}
\caption{}\label{odd case}
\end{figure}

We prove theorem for $A$ and the case $\mathring A$ can be proved in a similar vein.
Due to the previous proposition we assume that there is an interior symmetric $k$-periodic point $z$ of $f$ for $k\in 2\N-1$. Since $k$ is odd, one of $f^\ell(z)$, $1\leq\ell\leq k$ lies on $\Fix I_{A}$, we may assume $z\in\Fix I_{A}\cap\mathring A$. Denote by $g=f^k$. We consider the covering $S\to A$ and lift $I_A$ to $I_S$. Thus we can choose a lift $G:S\to S$ of $g:{A}\to { A}$ such that $G\circ I_S=I_S\circ G^{-1}$ and 
$$
\Fix (G\circ I_{ S})\cap \Fix I_{ S}\neq\emptyset.
$$ 
Suppose that the cardinality of the above set is finite since otherwise there are infinitely many symmetric $k$-periodic points. Since $f$ and hence $G$ is isotopic to the identity, $\Fix (G\circ I_S)$ is a line connecting the upper boundary and the lower boundary of $S$, see Lemma \ref{lem:fixd loci of involutions}. We choose the closed sub-segment $s$ of $\Fix (G\circ I_{ S})$ with boundary $s_\pm$ satisfying
$$
s\cap \Fix I_{ S}\cap\mathring S=\{s_-\},\quad s\cap (\R\x\{1\})=\{s_+\}.
$$
Note that $I_S(s)=G^{-1}(s)$, $s\cap I_{ S}(s)\cap\mathring S=\{s_-\}$, and therefore
\beq\label{eq1}
G^{m}(s)\cap G^{m+1}(s)\cap\mathring S=\{s_-\},\quad m\in\N\cup\{0\}.
\eeq
Without loss of generality, we assume $s\subset [0,\infty)\x[0,1]$.
Since $G(I_{S}(s_+))=s_+$ and $G$ is isotopic to the identity,  
\beq\label{eq2}
G^{\ell}(s_+)\in [s_+,\infty)\x\{1\},\quad \ell\in\N.
\eeq

Suppose that there is $q_0\in\N$ such that  $G^{q_0}(s)\cap\Fix I_S\cap \mathring S\neq\{s_-\}$. 
Then there exists
$$
(0,y_q)\in G^{q}(s)\cap\Fix I_S\cap\mathring S,\quad q\geq q_0
$$ 
with $y_q>y_{q+1}$ different from $s_-$ due to \eqref{eq1}, \eqref{eq2}, and $G^{q}(s_-)=s_-$. See the first picture of Figure \ref{odd case}.  Due to  Lemma \ref{lem:fix loci} and Lemma \ref{lem:intersection point=>symmetric periodic point}, this shows the existence of infinitely many interior symmetric odd-periodic points of $f$.

Suppose that $G^{\ell}(s)\cap\Fix I_S\cap \mathring S=\{s_-\}$ for every $\ell\in\N$. Let $D$ be the domain enclosed by $s\cup I_S(s)$ and $\R\x\{1\}$. Observe that 
$$
G^{\ell}(D)\cap G^{\ell+1}(D)=G^{\ell+1}(s),\quad \ell\in\N\cup\{0\}
$$
by \eqref{eq1}, see the second picture of Figure \ref{odd case}. Since $\bigcup_{\ell\in\N} G^\ell(D)\subset [0,\infty)\x[0,1]$ is connected and has infinite area, there exists $n(\ell)\in\N$, $\ell\in\N$ such that 
$$
\lim_{\ell\to\infty}n(\ell)=\infty,\quad G^{\ell}(s)\cap (\{n\pi\}\x[0,1])\cap\mathring S\neq\emptyset,\quad 0\leq\forall n\leq n(\ell).
$$
Since $\{n\pi\}\x[0,1]$'s are the lift of $\Fix I_A$, this shows that there are infinitely many interior symmetric odd-periodic points of $f$.
\end{proof}

\subsection{Proofs of Theorem \ref{thm:PB theorem} and Theorem \ref{thm:necessary and sufficient condition}}\quad\\[-1.5ex]

As we emphasized, we do {\em not} require here $f$ to be area preserving. Recall that we denote by $\Fix I_\Om=Y_1\sqcup\cdots\sqcup Y_n$ where $Y_i$ are disjoint intervals for an invariant connected possibly non-closed domain $\Om\subset\R^2$.
\begin{figure}[htb]
\includegraphics[width=.5\textwidth,clip]{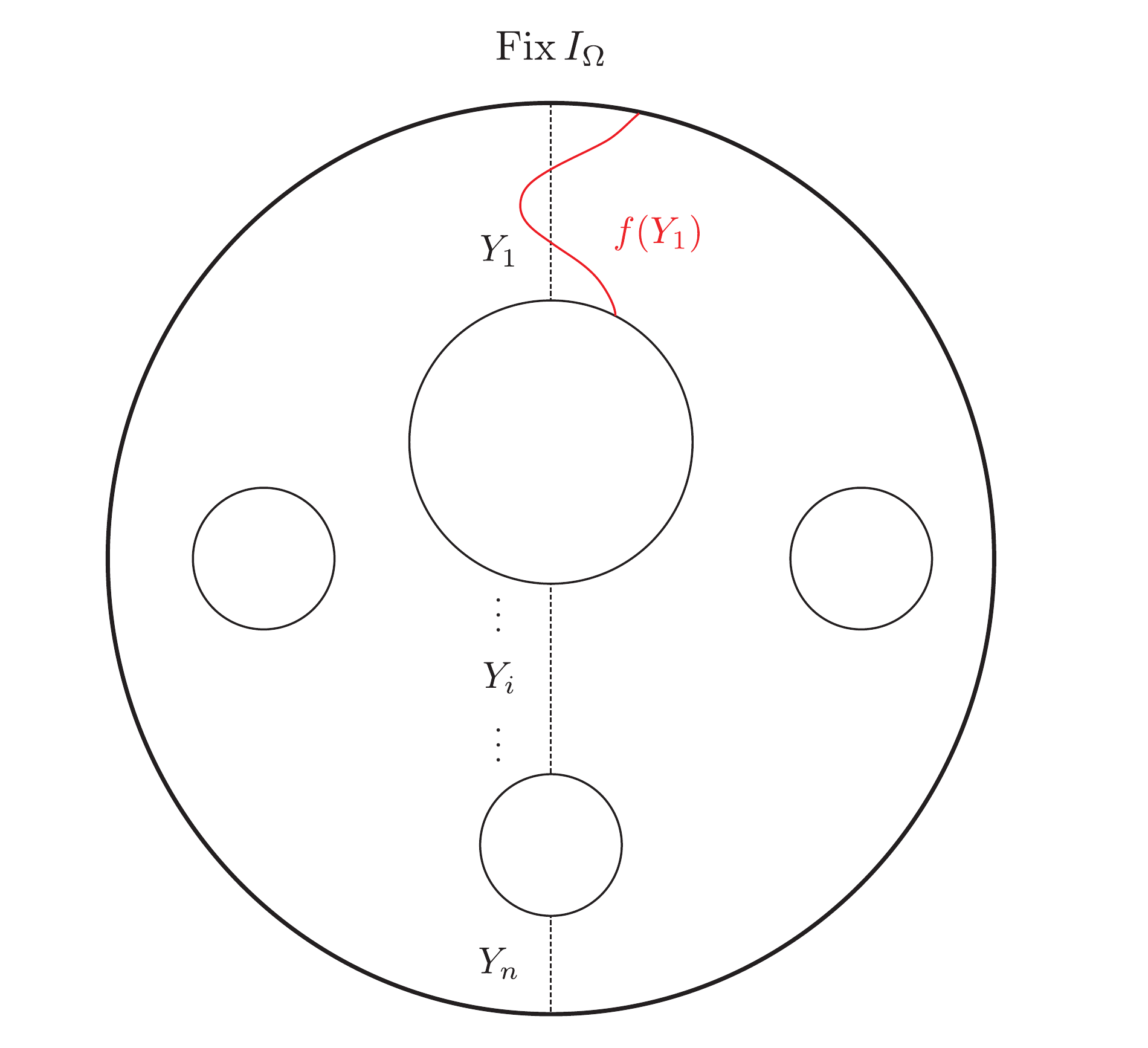}
\caption{}\label{iff}
\end{figure}
\begin{Thm}\label{thm:twist}
Let $f$ be a reversible map on $\Om$ isotopic to the identity. Then $f$ has a (an interior) symmetric fixed point on $Y_i$ if and only if $f(Y_i)\cap Y_i\neq\emptyset$ ($f(Y_i)\cap Y_i\cap\mathring \Om\neq\emptyset$) for  $i\in\{1,\dots,n\}$. 
\end{Thm}
\begin{proof}
See Figure \ref{iff}. If $f$ has a symmetric fixed point $z\in Y_i$,  $z=f(z)\in f(Y_i)\cap Y_i$ for $1\leq i\leq n$. For the converse, suppose that $z\in f(Y_i)\cap Y_i$ for $1\leq i\leq n$. Then $z=f(z_0)$ for some $z_0\in Y_i$. We assume $z\neq z_0$ since otherwise we are done. Since both $z_0$, $z\in Y_i\subset\Fix I_\Om$, we compute
$$
f(z)=f\circ I_\Om(z)=I_\Om\circ f^{-1}(z)=I_\Om(z_0)=z_0.
$$
This implies $z$ is a symmetric 2-periodic point on $Y_i$ with $f(z)\in Y_i$, and hence Proposition \ref{period 2 real periodic point=>real fixed point} guarantees the existence of a symmetric fixed point of $f$ in between $z$ and $f(z)$ and hence on $Y_i$. The assertion concerning an interior symmetric fixed point follows in the same way.
\end{proof}

Although Theorem \ref{thm:twist} directly implies the theorem below, we outline two more elementary proofs of this. The first proof makes use of the fixed locus of the involution $f\circ I_A$ while the second proof uses the fixed point index. 

\begin{Thm}
Let $f$ be a reversible map on $A$ isotopic to the identity with the boundary twist condition. Then there is a symmetric fixed point of $f$ on each connected component of $\Fix I_A$. 
\end{Thm}

\begin{figure}[htb]
\includegraphics[width=0.9\textwidth,clip]{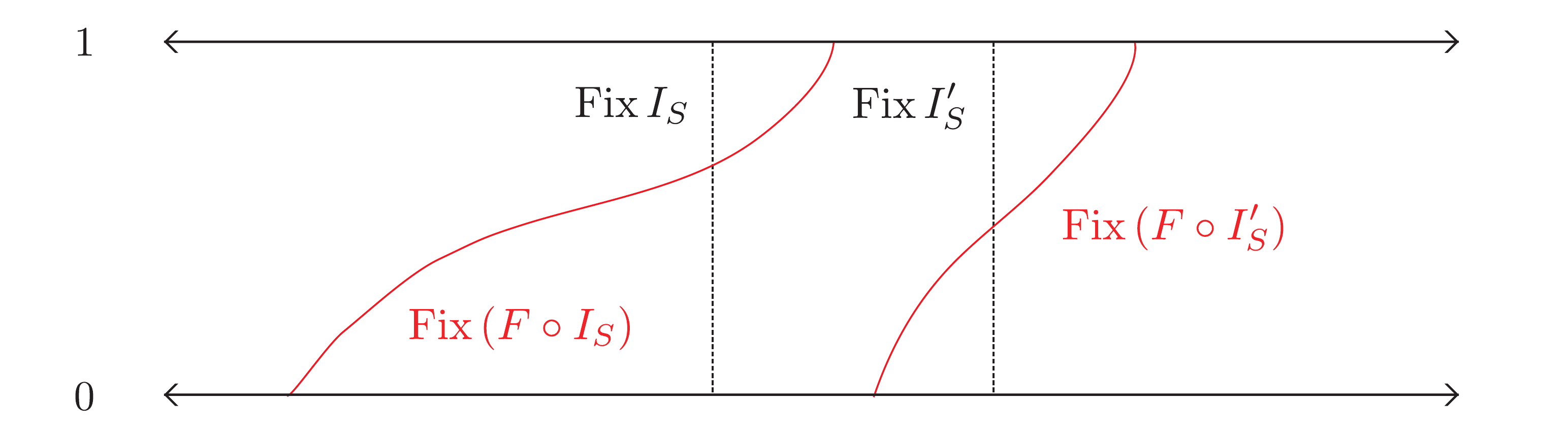}
\caption{}\label{fig:PB proof2 strip}
\end{figure}
\noindent\textsc{Sketch of Proof 1.}
See Figure \ref{fig:PB proof2 strip}. We lift $f:A\to A$ to a reversible map $$F=(F_1,F_2): S=\R\x[0,1]\to S$$ such that $F_1(0,1)\in(0,2\pi]$. Then by the boundary twist condition, 
$$
\Fix (F\circ I_S)\cap(\R\x\{1\})\in(0,\pi],\quad \Fix (F\circ I_S)\cap(\R\x\{0\})<0.
$$
Therefore we have 
$$
\Fix (F\circ I_S)\cap\Fix I_S\cap\mathring S\neq\emptyset,
$$
i.e. there is an interior symmetric fixed point of $f$ on one connected component of $\Fix I_A$. To find another interior symmetric fixed point of $f$ on the other connected component of $\Fix I_A$, consider the reflection $I_S':(x,y)\mapsto(2\pi-x,y)$ which is another lift of $I_A$ to $S$. Then $F\circ I'_S$ is also reversible and we have
$$
\Fix (F\circ I'_S)\cap\Fix I'_S\cap\mathring S\neq\emptyset.
$$
by the boundary twist condition again. This gives another symmetric fixed point of $f$.
\hfill$\square$\\[-1ex]

\begin{figure}[htb]
\includegraphics[width=.8\textwidth,clip]{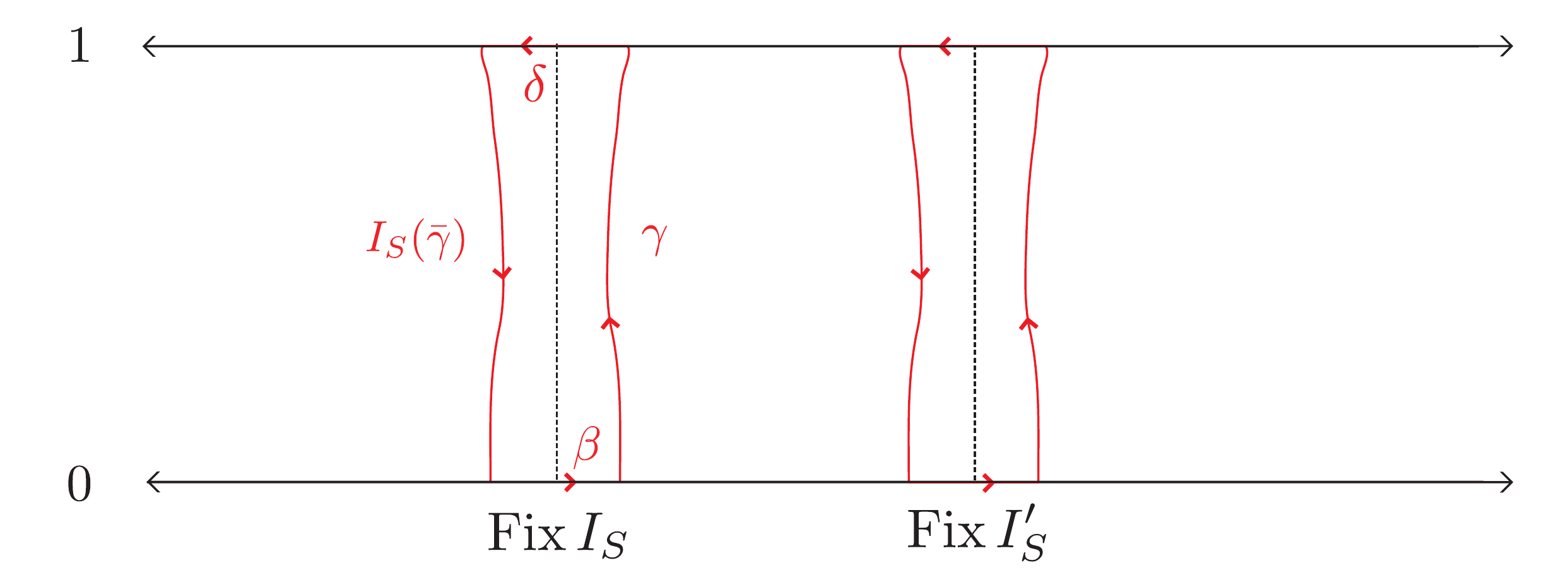}
\caption{}\label{PB proof 1}
\end{figure}
\noindent\textsc{Sketch of Proof 2.}
See Figure \ref{PB proof 1}.
Suppose that a lift $F:S\to S$ of $f:A\to A$ does not have any fixed points on $\Fix I_S$. Then there is no fixed point of $F$ inside $(-\epsilon,\epsilon)\x[0,1]$ for small  $\epsilon>0$. Now we choose any path $\gamma:[0,1]\to(0,\epsilon)\x[0,1]$ from $\R\x\{0\}$ to $\R\x\{1\}$. Then due to the boundary twist condition,
\beq\label{eq:half integer}
i(\gamma,F\circ\gamma)\in\frac{1}{2}+\Z.
\eeq
Moreover as in the proof of Proposition \ref{period 2 real periodic point=>real fixed point}, we have 
\beq\label{eq:same index}
 i(I_S\circ \bar\gamma,F\circ I_S\circ\bar\gamma)=i(\gamma,F\circ\gamma).
\eeq
We choose auxiliary paths $\beta$ and $\delta$ in the boundary of $S$ as in Figure \ref{PB proof 1} so that $\alpha:=\gamma*\delta*I(\bar\gamma)*\beta$ forms a loop containing $\Fix I_S$. Then using \eqref{eq:half integer} and \eqref{eq:same index} we compute
$$
i(\alpha,F\circ\alpha)=i(I\circ \bar\gamma,F\circ I\circ\bar\gamma)+i(\gamma,F\circ\gamma)\in1+2\Z.
$$
Therefore there is a fixed point of $F$ inside the domain enclosed by the loop $\alpha$. This contradiction shows the existence of a symmetric fixed point of $f$ on one connected component of $\Fix I_A$.
In a similar way, we are able to find a fixed point of $F$ on $\Fix I_S'$ which gives another symmetric fixed point of $f$ on the other connected component of $\Fix I_A$.
\hfill$\square$\\[-1ex]

A standard argument shows the following, see for instance \cite[Corollary 8.6]{MS98}.
\begin{Cor}
Let $f:A\to A$ be a reversible map isotopic to the identity with the boundary twist condition. Then for  each $\frac{p}{q}\in\Q$ satisfying
$$
\min_{x\in\R}(F_1(x,0)-x)<\frac{2\pi p}{q}< \max_{x\in\R}(F_1(x,1)-x),
$$
there is a symmetric $q$-periodic point of $f$ with rotation number $\frac{p}{q}$ on each connected component of $\Fix I_A$. 
\end{Cor}

\subsubsection*{Acknowledgments} {I would like to thank Urs Fuchs for helpful discussions. The author is supported by DFG grant KA 4010/1-1. This work is also partially supported by SFB 878-Groups, Geometry, and Actions.}

\end{document}